\documentclass{amsart}
\usepackage{amsfonts}
\usepackage{graphicx}
\usepackage{amscd}
\usepackage{amsmath}
\usepackage{amssymb}
\usepackage{stmaryrd} %produto kulkarni-nomizo
\usepackage{color}
\usepackage{comment}

\makeatletter
\@namedef{subjclassname@2010}{%
  \textup{2010} Mathematics Subject Classification}
\makeatother

\theoremstyle{plain}
\newtheorem*{acknowledgement}{Acknowledgement}

\newtheorem{corollary}{\bf Corollary}
\newtheorem{definition}{\bf Definition}
\newtheorem{lemma}{\bf Lemma}

\newtheorem{theorem}{\bf Theorem}

\theoremstyle{definition}

\numberwithin{equation}{section}

\title[Einstein type manifolds]{Compact Einstein-type manifolds \\ with parallel Ricci tensor}

\author{M. Andrade$^\ast$}
 
\author{H. Baltazar}

\author{C. Queiroz}

\address[M. Andrade]{Departamento de Matem\'{a}tica, Universidade Federal de Sergipe \\
49100-000, S\~ao Cristov\~ao, Sergipe, Brazil.}
\email{maria@mat.ufs.br}

\address[H. Baltazar]{Departamento de Matem\'{a}tica, Universidade Federal do Piau\'{\i}\\
64049-550 Te\-re\-si\-na, Piau\'{\i}, Brazil.}
\email{halyson@ufpi.edu.br}

\address[C. Queiroz]{Departamento de Matem\'{a}tica, Universidade Federal do Piau\'{\i}\\
64049-550 Te\-re\-si\-na, Piau\'{\i}, Brazil.}
\email{chrisqueiroz@ufpi.edu.br}

\thanks{$^\ast$ Corresponding author.}

\begin{document}

\newcommand{\spacing}[1]{\renewcommand{\baselinestretch}{#1}\large\normalsize}
\spacing{1.2}

\begin{abstract}
In this paper, we deduce a Bochner-type identity for compact gradient Einstein-type manifolds with boundary. As consequence, we are able to show a rigidity result for Einstein-type manifolds assuming the parallel Ricci curvature condition. Moreover, we provide a condition on the norm of the gradient of the potential function in order to classify such structures.
\end{abstract}

\maketitle

\section{Introduction}\label{intro}

The study of Einstein-type manifolds was introduced by Catino et al. \cite{CMMR} and emerges as a way to unify some structures that are widely investigated in the literature, namely, we can highlight the static vacuum Einstein equation \cite{HCY16}, static vacuum equation with non null cosmological constant \cite{Amb}, static perfect fluid \cite{Shen}, Miao-Tam equation \cite{MT11}, $(\lambda,n+1)$-Einstein manifolds \cite{FS20}, and many others. Now, following the terminology employed in \cite{AA22,HY,HY22,L21} we recall the definition of Einstein-type manifolds.

\begin{definition}\label{def1} 
An Einstein-type manifold is a Riemannian manifold $(M^{n},g)$ with $n\geq3$ which admits smooth functions $f,h:M\rightarrow\mathbb{R}$ such that 
\begin{equation}\label{fundEq}
fRic=Hessf+hg,
\end{equation}
where $f>0$ in $int(M)$ and $f^{-1}(0)=\partial M.$ Here, $Ric$ and $Hess$ stand, respectively, for the Ricci tensor and the Hesssian form on $M^{n}.$ 
\end{definition}

We shall refer equation (\ref{fundEq}) as the fundamental equation of an Einstein-type manifold $(M^{n},g,f,h).$ These structures generalize important equations. For example, if $f$ is constant, then $(M^n, g)$ is Einstein and $\partial M = \emptyset$. If we consider $h=0,$ then it reduces to the static vacuum Einstein equation. If $h=\frac{R}{n-1}f,$ then we have the vacuum static equation. If $h=\frac{R}{n-1}f+\dfrac{k}{n-1},$ where $k$ is a constant, we obtain the so called $V$-static equation. If $h=\frac{R}{n-1}f-\frac{R}{n(n-1)},$ we obtain the critical point equation with $f=1+F$ and $F$ is the potential function in the critical point equation. Finally, if $h=\frac{R-\rho-\mu}{n-1}f,$ we have the static perfect fluid equation. For more details see \cite{AA22}.

Recently, Freitas and Gomes were able to classify Einstein-type manifold which are Einstein, see Theorems 1 and 2 in \cite{GF22}. In their paper, the fundamental equation of the Einstein-type manifolds were studied in the following way
\begin{equation}\label{FGeq}
Hessf=\frac{\mu}{\beta}f(\Lambda g-\frac{\alpha}{\beta}Ric)+\gamma g,
\end{equation}
for some smooth function $\Lambda$ on $M^{n}$ and constants $\alpha,\beta,\mu,\gamma\in \mathbb{R},$ with $\beta\neq0.$ Clearly, considering $\frac{\mu\alpha}{\beta^{2}}=-1$ and $h=-\frac{\mu\Lambda}{\beta}f-\gamma,$ we derive the equation (\ref{fundEq}). Thus, the equation (\ref{fundEq}) can be seen as a particular case of (\ref{FGeq}).

One of the most important structure that have been studied in the literature and which satisfies the Einstein-type equation for $h=-\Delta f$, is the positive static triples. For our purpose,  let us recall the classification of positive static triples for a three dimensional manifold satisfying the Ricci parallel condition. In fact, as consequence of Kobayashi and Lafontaine results for static locally conformally flat manifolds, see \cite{kobayashi} and \cite{lafontaine}, respectively, it is immediate to deduce the following result.

\begin{theorem}[Kobayashi \cite{kobayashi}, Lafontaine \cite{lafontaine}]\label{KL}
Let $(M^{3},\,g,\,f)$ be an $3$-dimensional positive static triple with scalar curvature $R=6.$ Suppose that $(M^{3},\,g)$ has parallel Ricci tensor, then $(M^{3},\,g,\,f)$ is covered by a static triple equivalent to one of the following two static triples:
\begin{enumerate}
\item The standard hemisphere with canonical metric $(\mathbb{S}^{3}_{+},g_{\mathbb{S}^{3}}).$
\item The standard cylinder over $\mathbb{S}^{2}$ with the product metric
$$\Big(M=\Big[0,\frac{\pi}{\sqrt{3}}\Big]\times\mathbb{S}^{2},\; g=dt^{2}+\frac{1}{3}g_{\mathbb{S}^{2}}\Big).$$
\end{enumerate}
\end{theorem}

Before to proceed, let us define the function $\lambda$ in terms of $h$ and the potential function $f$, as follows
\begin{equation}\label{lambda}
\lambda=(n-1)h-Rf=-\dfrac{(n-1)}{n}\left[\dfrac{fR}{n-1}+\Delta f\right].
\end{equation}

Such a function appear in \cite{HY} and the authors were able to deduce that the scalar curvature is constant if and only if the $\lambda$ is constant, see Proposition 2.1 in \cite{HY}.

In this article, we will able to classify Einstein-type manifolds satisfying the Ricci parallel curvature condition and the main ingredient will be to prove that this structures are in fact an $V$-static metrics. More precisely, we have established the following result.
\begin{theorem}\label{THM1Ricciparallel}
Let $(M^{n},g,f,h)$ be an Einstein-type manifold with parallel Ricci tensor. Then, the function $\lambda$ defined above is constant and one of the following assertions holds: 
\begin{itemize} 
\item[(1)] if $\lambda>0,$ then $(M^{n},g)$ is isometric to a geodesic ball in $\mathbb{R}^{n}$, $\mathbb{S}^{n}$, or $\mathbb{H}^{n}.$
\item[(2)] if $\lambda=0$ and $n=3$, then $(M^{3},g)$ is covered by one of the static triples described in Theorem~\ref{KL};
\item[(3)] if $\lambda<0$ and $M^{n}$ has connected boundary, then $(M^{n},g)$ is isometric to a geodesic ball in $\mathbb{S}^{n}$. 
\end{itemize} 
\end{theorem}

Our next result was inspired by B. Leandro \cite{L15}, where the author studied a necessary and sufficient condition on the norm of the gradient of the potential function for a CPE metrics to be Einstein, and more recently by first author and C. Queiroz in \cite{BQ24}, where similar results were deduced for Miao-Tam critical metrics. More precisely, we have the following result.

\begin{theorem}\label{Thm2}
Let $(M^{n},g,f,h)$  be an Einstein-type manifold satisfying 
\begin{equation}\label{idenTHM2}
\frac{1}{2}|\nabla f|^{2}+\frac{1-2n}{2n(n-1)}Rf^{2}+hf=const.
\end{equation}
Then, $(M^{n},g)$ is an Einstein manifold provided
$$\int_{M}\langle\nabla R,\nabla f\rangle dM_{g}\geq0.$$
\end{theorem}

As an immediate corollary, we have that 
\begin{corollary}
Let $(M^{n},g,f,h)$  be an Einstein-type manifold with constant scalar curvature satisfying 
\begin{equation}
\frac{1}{2}|\nabla f|^{2}+\frac{1-2n}{2n(n-1)}Rf^{2}+hf=const.
\end{equation}
Then, the function $\lambda$ defined in (\ref{lambda}) is constant and one of the following assertions holds: 
\begin{itemize} 
\item[(1)] if $\lambda>0,$ then $(M^{n},g)$ is isometric to a geodesic ball in $\mathbb{R}^{n}$, $\mathbb{S}^{n}$, or $\mathbb{H}^{n}.$
\item[(2)] if $\lambda=0,$  then, up to a finite quotient, $(M^{n},g)$ is isometric to a standard hemisphere with canonical metric $(\mathbb{S}^{n}_{+},g_{\mathbb{S}^{n}}).$
\item[(3)] if $\lambda<0$ and $M^{n}$ has connected boundary, then $(M^{n},g)$ is isometric to a geodesic ball in $\mathbb{S}^{n}$; 
\end{itemize} 
\end{corollary}

Moreover, restricted to three-dimensional case, the integral condition can be weakened. In fact, we have the following result:

\begin{corollary}\label{Co2}
Let $(M^{3},g,f,h)$  be an Einstein-type manifold satisfying 
\begin{equation}
\frac{1}{2}|\nabla f|^{2}-\frac{5}{12}Rf^{2}+hf=const.
\end{equation}
Then, $(M^{3},g)$ is an Einstein manifold provided
$$\int_{\partial M}\langle\nabla R,\nabla f\rangle dS\geq0.$$
\end{corollary}

\section{Preliminaries}
\label{Preliminaries}

In this section we need recall some special tensors which will be important for understanding the desired results. Consider a Riemannian manifold $(M^n,g)$ of dimension $n\geq 3.$ We start with Weyl tensor, given by the decomposition formula
\begin{equation}\label{weyl}
W_{ijkl}=R_{ijkl}-\frac{1}{n-2}(Ric\varowedge g)_{ijkl}+\frac{R}{2(n-1)(n-2)}(g\varowedge g)_{ijkl},
\end{equation}
where $R_{ijkl}$ stands for the Riemann curvature tensor. The symbol $\varowedge$, in the above expression, denotes the Kulkarni-Nomizu product,  which is defined for any two symmetric $(0,2)$-tensor $S$ and $T$ as follows
$$(S\varowedge T)_{ijkl}=S_{ik}T_{jl}+S_{jl}T_{ik}-S_{il}T_{jk}-S_{jk}T_{il}.$$

In the sequel, we have the Cotton tensor $C$ which is given by
\begin{equation}\label{cotton}
\displaystyle{C_{ijk}=\nabla_{i}R_{jk}-\nabla_{j}R_{ik}-\frac{1}{2(n-1)}\big(\nabla_{i}Rg_{jk}-\nabla_{j}R g_{ik}).}
\end{equation}
It is easy to check that $C_{ijk}$ is skew-symmetric in the first two indices and trace-free in any two indices. These two tensors described above are related as follows
\begin{equation}\label{cottonwyel}
\displaystyle{C_{ijk}=-\frac{(n-2)}{(n-3)}\nabla_{l}W_{ijkl},}
\end{equation}provided $n\ge 4.$

Moreover, it is important to remember that, for any Riemannian manifold $M^{n},$  we have that
\begin{equation}
\nabla_{i}\nabla_{j}R_{kl}-\nabla_{j}\nabla_{i}R_{kl}=R_{ijkm}R_{ml}+R_{ijlm}R_{km}.
\end{equation}
For more details see \cite{Chow}. Such commutation formula for the first covariant derivative of the Ricci curvature can be used to obtain an interesting formula for the Laplacian of the norm of the Ricci tensor over an arbitrary manifold. In fact, with a standard computation, we may deduce
\begin{eqnarray}\label{Lric}
\Delta|Ric|^{2}&=&2|\nabla Ric|^{2}-|C_{ijk}|^{2}+2\nabla_{i}(C_{ijk}R_{jk})+2(R_{ij}R_{ik}R_{jk}-R_{ik}R_{jl}R_{ijkl})\nonumber\\
&&-\frac{n}{2(n-1)}|\nabla R|^{2}+\frac{1}{n-1}div[(n-2)Ric(\nabla R)+R\nabla R].
\end{eqnarray}
We refer the reader to \cite[Lemma 2.1]{bbb}, for its proof. Furthermore, it is possible to relate the tensor curvature with a smooth function $f\in C^{\infty}(M)$ using the Ricci identity,
\begin{eqnarray}\label{Ricciid}
\nabla_i\nabla_j\nabla_kf-\nabla_j\nabla_i\nabla_kf=R_{ijkl}\nabla_{l}f.
\end{eqnarray}

Now, we note that the fundamental equation of an Einstein-type manifold can be rewritten in the tensorial language as follows
\begin{equation}\label{eq:tensorial}
fR_{ij}=\nabla_{i}\nabla_{j}f+hg_{ij}.
\end{equation}
Tracing (\ref{eq:tensorial}), we have
\begin{equation}\label{eqtrace}
fR=\Delta f+nh.
\end{equation}
Using \eqref{eq:tensorial}, we infer that
$$\nabla_ifR_{jk}+f\nabla_iR_{jk}=\nabla_i\nabla_j\nabla_kf+\nabla_ihg_{jk}.$$
From \eqref{Ricciid}, we obtain
$$\nabla_ifR_{jk}+f\nabla_iR_{jk}=\nabla_j\nabla_i\nabla_kf+R_{ijkl}\nabla_{l}f+\nabla_ihg_{jk}.$$
We can conclude from the contracted second Bianchi identity $\frac{1}{2}\nabla_iR=g^{jk}\nabla_jR_{ki}$ and contracting on $i$ and $k$, we get
\begin{eqnarray*}
\frac{1}{2}f\nabla_jR= \nabla_j\Delta f+ \nabla_jh. 
\end{eqnarray*}
From \eqref{eqtrace}, we conclude that
\begin{eqnarray}\label{gradh}
(n-1)\nabla_jh=R\nabla_jf+\frac{1}{2}f\nabla_jR.
\end{eqnarray}
This implies that
\begin{eqnarray}\label{hessh}(n-1)\nabla_i\nabla_jh=\nabla_iR\nabla_jf+R\nabla_i\nabla_jf+\frac{1}{2}\nabla_if\nabla_jR+\frac{1}{2}f\nabla_i\nabla_jR.
\end{eqnarray}
Therefore, the trace of \eqref{hessh} on $i$ and $j$ yields
\begin{eqnarray}\label{laplah}
(n-1)\Delta h =\frac{3}{2}\langle \nabla R,\nabla f\rangle +R\Delta f+\frac{1}{2}f\Delta R.
\end{eqnarray}

Furthermore, by using (\ref{eqtrace}) it is not difficult to check that
\begin{equation}\label{IdRicHess} 
f\mathring{Ric}=\mathring{Hess f},
\end{equation}
where $\mathring{T}$ stands for the traceless of tensor $T.$

Under this notation we get the following formula for an Einstein-type manifold
\begin{equation}\label{auxC}
fC_{ijk}=R_{ijkl}\nabla_{l}f+\frac{R}{n-1}(\nabla_{i}fg_{jk}-\nabla_{j}fg_{ik})-(\nabla_{i}fR_{jk}-\nabla_{j}fR_{ik}).
\end{equation} 
Its proof can be found in \cite[Lemma 1]{L21}.

\section{A Bochner formula for Einstein-type manifolds}
In this section we shall deduce a couple of divergence formulas,  which allows us to obtain the classification of the Einstein-type manifolds under Ricci parallel condition. Such formulas were inspired by recent works considering the well-known structures, namely,  CPE metrics, Miao-Tam critical metrics and positive static triples, for more details, see for instance \cite{BaltCPE} and \cite{balt18}. 

\begin{lemma}\label{auxint1}
Let $(M^{n},g,f,h)$ be an Einstein-type manifold. Then we have
\begin{eqnarray*}
{\rm div} X&=&\nabla_{i}(fC_{ijk}R_{jk})+\left(\frac{Rf}{n-1}+\Delta f\right)|\mathring{Ric}|^{2}+\frac{1}{n-1}Ric(\nabla R,\nabla f)\\
&&+\frac{n-4}{4n(n-1)}\langle\nabla R^{2},\nabla f \rangle+\langle \nabla|\mathring{Ric}|^{2},\nabla f\rangle,
\end{eqnarray*}
where $X_{i}=R_{ik}R_{kj}\nabla_{j}f+R_{ijkl}\nabla_{l}fR_{jk}.$
\end{lemma}
\begin{proof}
We start substituting Eq. (\ref{auxC}) into ${\rm div} X$, i.e., 
\begin{eqnarray*}
{\rm div}X&=&\nabla_{i}\Big[R_{ik}R_{kj}\nabla_{j}f+fC_{ijk}R_{jk}-\frac{R}{n-1}(\nabla_{i}fg_{jk}-\nabla_{j}fg_{ik})R_{jk}\\
&&+(\nabla_{i}fR_{jk}-\nabla_{j}fR_{ik})R_{jk} \Big]\\
&=&\nabla_{i}\Big[fC_{ijk}R_{jk}-\frac{R^{2}}{n-1}\nabla_{i}f+\frac{R}{n-1}R_{ij}\nabla_{j}f+\nabla_{i}f|Ric|^{2}\Big]\\
&=&\nabla_{i}(fC_{ijk}R_{jk})-\frac{3}{4(n-1)}\langle\nabla R^{2},\nabla f\rangle+\left(|Ric|^{2}-\frac{R^{2}}{n-1}\right)\Delta f\\
&&+\frac{1}{n-1}Ric(\nabla R,\nabla f)+\frac{R}{n-1}R_{ij}\nabla_{i}\nabla_{j}f+\langle\nabla f,\nabla|Ric|^{2}\rangle,
\end{eqnarray*}
which can be rewritten using (\ref{eq:tensorial}), in the following way
\begin{eqnarray*}
{\rm div}X&=&\nabla_{i}(fC_{ijk}R_{jk})+\left(\frac{Rf}{n-1}+\Delta f\right)|\mathring{Ric}|^{2}+\frac{1}{n-1}Ric(\nabla R,\nabla f)\\
&&+\frac{R^{2}}{n(n-1)}(-\Delta f-nh+Rf) +\frac{n-4}{4n(n-1)}\langle \nabla R^{2},\nabla f\rangle+\langle \nabla f,\nabla|\mathring{Ric}|^{2}\rangle.
\end{eqnarray*}
To conclude, just apply (\ref{eqtrace}).
\end{proof}

Now, we will provide another formula for ${\rm div} X.$

\begin{lemma}\label{auxint2}
Let $(M^{n},g,f,h)$ be an Einstein-type manifold. Then we have
\begin{eqnarray*}
{\rm div} X&=&f(R_{ij}R_{ik}R_{jk}+R_{ijkl}R_{il}R_{jk})+\frac{n}{2(n-1)}Ric(\nabla R,\nabla f)+\frac{f}{2}|C_{ijk}|^{2}\\
&&+C_{ijk}R_{ik}\nabla_{j}f-\frac{1}{2n(n-1)}\langle\nabla R^{2},\nabla f\rangle+\frac{1}{2}\langle\nabla f,\nabla|\mathring{Ric}|^{2}\rangle,
\end{eqnarray*}
where $X_{i}=R_{ik}R_{kj}\nabla_{j}f+R_{ijkl}\nabla_{l}fR_{jk}.$
\end{lemma}
\begin{proof}
By direct computation using the fundamental equation of an Einstein-type manifold, we obtain
\begin{eqnarray}\label{auxdivX}
{\rm div}X&=&\nabla_{i}\nabla_{j}fR_{ik}R_{jk}+\frac{1}{2}Ric(\nabla R,\nabla f)+\nabla_{j}fR_{ik}\nabla_{i}R_{jk}\nonumber\\
&&+R_{ijkl}\nabla_{i}\nabla_{l}fR_{jk}+\nabla_{i}R_{ijkl}\nabla_{l}fR_{jk}+R_{ijkl}\nabla_{l}f\nabla_{i}R_{jk}\nonumber\\
&=&f(R_{ij}R_{ik}R_{kj}+R_{ijkl}R_{il}R_{jk})+\frac{1}{2}Ric(\nabla R,\nabla f)-\frac{1}{2}\langle\nabla f,\nabla|Ric|^{2}\rangle\nonumber\\
&&+2\nabla_{j}fR_{ik}\nabla_{i}R_{jk}+R_{ijkl}\nabla_{l}f\nabla_{i}R_{jk},
\end{eqnarray}
where in the last equality we have used the once contracted second Bianchi identity.

Now, using (\ref{cotton}) and  (\ref{auxC}) we can rewrite (\ref{auxdivX}) as 
\begin{eqnarray}\label{auxdivX1}
{\rm div}X&=&f(R_{ij}R_{ik}R_{kj}+R_{ijkl}R_{il}R_{jk})+\frac{1}{2}Ric(\nabla R,\nabla f)+\frac{f}{2}|C_{ijk}|^{2}\nonumber\\
&&-\frac{1}{4(n-1)}\langle \nabla f,\nabla R^{2}\rangle+\nabla_{i}R_{jk}R_{ik}\nabla_{j}f
\end{eqnarray}
Finally, it sufficient to substitute (\ref{cotton}) into (\ref{auxdivX1}) to get the requested result.
\end{proof}

\begin{lemma}\label{LBochner}
Let $(M^{n},g,f,h)$ be an Einstein-type manifold. Then we have:
\begin{eqnarray*}
{\rm div}(X)&=& -\left(\frac{Rf}{n-1}+\Delta f\right)|\mathring{Ric}|^{2}-\langle\nabla f,\nabla|\mathring{Ric}|^{2} \rangle+\frac{n-2}{n-1}\mathring{Ric}(\nabla R,\nabla f)\\
&&+f\left(|C_{ijk}|^{2}-|\nabla Ric|^{2}+\frac{n}{4(n-1)}|\nabla R|^{2}\right),
\end{eqnarray*}
where $X_{i}=-\frac{f}{2}\nabla_{i}|\mathring{Ric}|^{2}+2fC_{ijk}R_{jk}+\frac{n-2}{2(n-1)}fR_{ij}\nabla_{j}R-\frac{n-2}{4n(n-1)}f\nabla_{i}R^{2}$
\end{lemma}
\begin{proof}
First, comparing Lemma~\ref{auxint1} and Lemma~\ref{auxint2} we may deduce 
\begin{eqnarray}\label{a1}
&&f(R_{ij}R_{ik}R_{jk}+R_{ijkl}R_{il}R_{jk})\nonumber\\
&&=\nabla_{i}(fC_{ijk}R_{jk})-C_{ijk}R_{ik}\nabla_{j}f+\left(\frac{Rf}{n-1}+\Delta f\right)|\mathring{Ric}|^{2}+\frac{1}{2}\langle\nabla f, \nabla |\mathring{Ric}|^{2}\rangle\nonumber\\
&&-\frac{n-2}{2(n-1)}Ric(\nabla R,\nabla f)+\frac{n-2}{4n(n-1)}\langle\nabla R^{2},\nabla f\rangle-\frac{f}{2}|C_{ijk}|^{2}.
\end{eqnarray}
On the other hand, with a straightforward computation, it is immediate to check that Eq. (\ref{Lric}) can be rewrite as 
\begin{eqnarray}\label{a2}
&&f(R_{ij}R_{ik}R_{jk}+R_{ijkl}R_{il}R_{jk})\nonumber\\
&&=\nabla_{i}\left(\frac{f}{2}\nabla_{i}|\mathring{Ric}|^{2}-fC_{ijk}R_{jk}+\frac{n-2}{2(n-1)}f\left(\frac{1}{2n}\nabla R^{2}-R_{ij}\nabla_{j}R\right)\right)\nonumber\\
&&-\frac{1}{2}\langle\nabla f,\nabla|\mathring{Ric}|^{2}\rangle+C_{ijk}R_{jk}\nabla_{i}f-f|\nabla Ric|^{2}+\frac{f}{2}|C_{ijk}|^{2}\nonumber\\
&&+\frac{n-2}{2(n-1)}\left(Ric(\nabla f,\nabla R)-\frac{1}{2n}\langle \nabla f,\nabla R^{2}\rangle\right)+\frac{n}{4(n-1)}f|\nabla R|^{2}.
\end{eqnarray}
To finalize, just compare the expressions (\ref{a1}) and (\ref{a2}).
\end{proof}

\section{Proof of Theorem~\ref{THM1Ricciparallel}}

\subsection{Proof of Theorem~\ref{THM1Ricciparallel}}
To begin with, we can use (\ref{eq:tensorial}) and (\ref{eqtrace}) to deduce
\begin{eqnarray}\label{vstatic}
-\Delta f g+Hessf-fRic&=&(-\Delta f-h)g\nonumber\\
&=&((n-1)h-Rf)g\nonumber\\
&=&\lambda g,
\end{eqnarray}
where $\lambda$ was defined in (\ref{lambda}).

Now, since we are assuming the parallel Ricci curvature condition, implies that $M^{n}$ has constant scalar curvature and the function $\lambda$ is constant, as consequence, identity (\ref{vstatic}) say for us that $(M^{n},g,f)$ is a $V$-static metric.

Therefore, if $\lambda>0,$ up to normalization, we have the Miao-Tam critial metrics introduced by Miao and Tam in \cite{MT09,MT11}, and the classification for Ricci parallel case was obtained by Baltazar and Ribeiro Jr. in \cite{br17}. Namely, we are in position to use Corollary 1 in \cite{br17} to conclude that $(M^{n},g)$ is isometric to a geodesic ball in $\mathbb{R}^{n}$, $\mathbb{S}^{n}$, or $\mathbb{H}^{n}.$ Otherwise, if $\lambda=0$ we have the well-known static spaces and the classification follows directly of Theorem~\ref{KL} announced in introduction. Finally, if $\lambda<0,$ with a straightforward computation, we have that  
$$-\left(\frac{Rf}{n-1}+\Delta f\right)=\frac{n}{n-1}\lambda,$$
and taking into account Lemma~\ref{LBochner} and our assumption that $(M^{n},g)$ has Ricci parallel curvature, we immediately have
$$\frac{n}{n-1}\lambda|\mathring{Ric}|^{2}=0$$
and this forces $(M^{n},g)$ to be Einstein. Now, it is important to note that the V-static equation satisfies  
\begin{eqnarray*}
Hessf&=&fRic+\Delta g+\lambda g\\
&=&fRic+\left(-\frac{Rf}{n-1}-\frac{n}{n-1}\lambda\right) g+\lambda g\\
&=&fRic-\frac{R}{n-1}fg-\frac{\lambda}{n-1}g,
\end{eqnarray*}
which is Eq. (\ref{FGeq}) for 
$$\frac{\mu\alpha}{\beta^{2}}=-1,\;\;\;\;\frac{\mu\Lambda}{\beta}=\frac{-R}{n-1}\;\;\;\;\text{and}\;\;\;\;\gamma=\frac{-\lambda}{n-1}.$$
Since we already know that $(M^n,g)$ is Einstein and we are assuming connect boundary, it suffices to apply Theorem 1 in \cite{GF22} to conclude that $M^{n}$ is isometric to a geodesic ball in $\mathbb{S}^{n}.$

\section{Proof of Theorem~\ref{Thm2}}

Firstly, motivated by the results in \cite{L15} we will find the condition for $(M^n, g,f,h)$ an Einstein-type manifold to be an Einstein manifold. For do this, we use  \eqref{IdRicHess} and \eqref{eqtrace} to obtain
$$f\left(R_{ij}-\frac{R}{n}g_{ij}\right)=\nabla_i\nabla_jf+\left(-\frac{fR}{n}+h\right)g_{ij}.$$

Thus, it follows that
\begin{eqnarray*}
f\mathring{R_{ij}}\nabla_{j}f&=&\nabla_i\nabla_jf\nabla_{j}f+\left(-\frac{fR}{n}+h\right)\nabla_if\\
&=&\frac{1}{2}\nabla_i|\nabla f|^2+\nabla_i(hf)-f\nabla_ih-\frac{fR}{n}\nabla_if
\end{eqnarray*}
Now, from \eqref{gradh}, we conclude that
\begin{eqnarray}\label{eqconst}
f\mathring{R_{ij}}\nabla_{j}f&=&\nabla_i\left(\frac{1}{2}|\nabla f|^2+hf+c_nRf^2\right)+\frac{f^2}{2n}\nabla_iR,
\end{eqnarray}
where $c_n=(1-2n)/(2n(n-1)).$
It is easy to see that if $(M^n,g)$ is an Einstein manifold, then 
$$\frac{1}{2}|\nabla f|^2+hf+c_nRf^2=\Lambda$$ is constant. Now, we study the function $\Lambda$ to obtain some classification for Einstein-type manifolds. 

\begin{lemma}\label{lemmain}
Let $(M^n, g,f,h)$ be an Einstein-type manifold. We consider the function $$\Lambda=\frac{1}{2}|\nabla f|^2+hf+c_nRf^2.$$ Then, 
$${\rm div}(f\nabla \Lambda)-2\langle \nabla \Lambda,\nabla f\rangle=f^3|\mathring{Ric}|^2+\frac{1}{2}f^2\langle \nabla R,\nabla f\rangle-\frac{1}{2n}{\rm div}(f^3\nabla R).$$
\end{lemma}

\begin{proof}
We assume that $(M^n,g,f,h)$ is an Einstein-type and $\Lambda$ is defined above. Under these conditions, we infer 
\begin{eqnarray}\label{eq0}
\nabla_k\Lambda&=&\frac{1}{2}\nabla_k|\nabla f|^2+\nabla_k(hf)+c_n\nabla_k(Rf^2).
\end{eqnarray}
Taking the second derivative of \eqref{eq0} and using \eqref{gradh}, we have
\begin{eqnarray*}
\nabla_i\nabla_k\Lambda&=&\frac{1}{2}\nabla_i\nabla_k|\nabla f|^2+\frac{1}{n-1}\left(R\nabla_if\nabla_kf+\frac{1}{2}f\nabla_if\nabla_kR\right)+f\nabla_i\nabla_kh\\
&&+\frac{1}{n-1}\left(R\nabla_if\nabla_kf+\frac{1}{2}f\nabla_kf\nabla_iR\right)+h\nabla_i\nabla_kf\\
&&+c_n(2f\nabla_if\nabla_kR+f^2\nabla_i\nabla_kR+2R\nabla_if\nabla_kf+2f\nabla_iR\nabla_kf+2fR\nabla_i\nabla_kf).
\end{eqnarray*}
 Tracing on $i$ and $k$ we arrive at
\begin{eqnarray}\label{eq1}
\Delta \Lambda&=&\frac{1}{2}\Delta|\Delta f|^2+\frac{R}{n(n-1)}|\nabla f|^2+\frac{(-3n+2)}{n(n-1)}f\langle \nabla f,\nabla R\rangle+f\Delta h++c_nf^2\Delta R\nonumber \\
&&+(2c_nfR+h)\Delta f.
\end{eqnarray}
Now, we recall the classical Bochner formula
$$\frac{1}{2}\Delta{|\nabla f|^2}=|Hessf|^2+Ric(\nabla f, \nabla f)+\langle\nabla f,\nabla(\Delta f)\rangle.$$
Thus, substituting \eqref{eqtrace} and \eqref{gradh}, we obtain
\begin{eqnarray}\label{eq2}
\frac{1}{2}\Delta{|\nabla f|^2}&=&|Hessf|^2+Ric(\nabla f, \nabla f)-\frac{R}{n-1}|\nabla f|^2\nonumber\\
&&+\frac{n-2}{2(n-1)}f\langle\nabla R,\nabla f\rangle.
\end{eqnarray}
Proceeding, we combine \eqref{eq1} with \eqref{laplah} and \eqref{eq2} to get
\begin{eqnarray*}
\Delta \Lambda&=&|Hessf|^2+Ric(\nabla f, \nabla f)-\frac{R}{n}|\nabla f|^2+\frac{n-4}{2n}f\langle\nabla R,\nabla f\rangle\\
&&+\left(\frac{Rf}{n-1}+h+2c_nfR\right)\Delta f-\frac{1}{2n}f^2\Delta R.
\end{eqnarray*}
From \eqref{eqtrace}, \eqref{IdRicHess} and  \eqref{eqconst}, we infer
\begin{eqnarray*}
f\Delta \Lambda=f^3|\mathring{Ric}|^2+\langle\nabla \Lambda,\nabla f\rangle+\frac{(n-1)}{2n}f^{2}\langle \nabla f,\nabla R\rangle -\frac{f^{2}}{2n}{\rm div}(f\nabla R).
\end{eqnarray*}
Using that ${\rm div}(f^3\nabla R)=2f^2\langle \nabla f,\nabla R\rangle+f^2{\rm div}(f\nabla R)$, we finish the proof.
\end{proof}

An immediate consequence of Lemma \ref{lemmain} is the following

\begin{corollary}
Let $(M^n, g,f,h)$ be an Einstein-type manifold. If $\displaystyle\int_Mf^2\langle\nabla R,\nabla f\rangle dM_{g}\geq 0$ and the function $\Lambda$ is constant along of the flow of $\nabla f$. Then, $(M^n, g)$ is Einstein.
\end{corollary}
\begin{proof}
In fact, since we are admitting that $\langle \nabla \Lambda, \nabla f\rangle=0$ and $\displaystyle\int_M f^2\langle\nabla R,\nabla f\rangle dM_{g}\geq 0,$ we integrating the expression in Lemma \ref{lemmain} and using that $f>0$ on int(M) and $f=0$ on $\partial M$, to conclude that $\mathring{Ric}=0,$ i.e. $(M^n, g)$ is Einstein.
\end{proof}
In particular, if we consider $f=1+F,$ where $F$ is the potential functional in the critical point equation we obtain the Corollary 1 in \cite{BF15}, because $R$ is constant.

\begin{corollary}
Let $(M^n, g,f,h)$ be an Einstein-type manifold. If $Ric(\nabla f)=\frac{R\nabla f}{n}$ and $\displaystyle\int_Mf^2\langle\nabla R,\nabla f\rangle dM_{g}\geq 0$. Then, $(M^n, g)$ is Einstein.
\end{corollary}
\begin{proof}
From \eqref{eqconst}, we obtain that $\langle \nabla \Lambda, \nabla f\rangle=-\frac{f^2}{2n}\langle \nabla R, \nabla f\rangle.$ Using Lemma \ref{lemmain}, we infer 
 $${\rm div}(f\nabla \Lambda)=f^3|\mathring{Ric}|^2+\frac{(n-2)}{2n}f^2\langle \nabla R,\nabla f\rangle-\frac{1}{2n}{\rm div}(f^3\nabla R).$$
Again, integrating it over $M$ and using that $f>0$ on int(M) and $f=0$ on $\partial M$ we finish the proof.
\end{proof}

\subsection{Proof of Theorem~\ref{Thm2}} Since $\Lambda$ is constant, from Lemma \ref{lemmain} and $f>0$ on $M$, we obtain 
\begin{equation}\label{divfnablaR}
{\rm div}(f\nabla R)=2nf|\mathring{Ric}|^2+(n-2)\langle \nabla R,\nabla f\rangle.
\end{equation}
Integrating it over $M,$ using that $\displaystyle\int_M\langle \nabla R, \nabla f \rangle dM_{g}\geq 0$ and $f=0$ on $\partial M,$ we conclude that $\mathring{Ric}=0,$ i.e., $(M^n,g)$ is Einstein. 

\subsection{Conclusion of the proof of Corollary~\ref{Co2}}
Before to proceed, notice that Eq. (\ref{divfnablaR}) can be rewritten as 
$$f\Delta R=2nf|\mathring{Ric}|^2+(n-3)\langle \nabla R,\nabla f\rangle.$$
Then, considering the three dimensional case, we immediately have
$$\Delta R=6|\mathring{Ric}|^2.$$
Whence, on integrating this last expression over $M^{3}$, we apply Stokes's formula to arrive at
$$\int_{\partial M}\langle\nabla R,\frac{\nabla f}{|\nabla f|}\rangle dS+6\int_{M}|\mathring{Ric}|^2dM_{g}=0,$$
where the normal vector on $\partial M$ is defined by $\nu=-\frac{\nabla f}{|\nabla f|}.$
Since $|\nabla f||_{\partial M}$ is constant and $\int_{\partial M}\langle\nabla R,\nabla f\rangle dS\geq0,$ we obtain an Einstein manifold. So, the proof is completed.

Moreover, we can deduce the following:

\begin{corollary}\label{ricautov}
Let $(M^{n},g,f,h)$  be an Einstein-type manifold with constant scalar curvature and suppose that $Ric(\nabla f)=\frac{R}{n}\nabla f.$ Then, $(M^n, g)$ is Einstein.
\end{corollary}
\begin{proof}In fact, since $R$ is constant, $Ric(\nabla f)=\frac{R}{n}\nabla f$ and $f>0,$ then by \eqref{eqconst}, we obtain that $\Lambda$ is constant. This result follows by Theorem \ref{Thm2}.
\end{proof}

Another consequence of Theorem \ref{Thm2} is the Theorem 1 proved by Leandro \cite{L15}, which give a necessary and sufficient condition for the Besse Conjecture or CPE equation to be true in terms of the norm of potential function $F$. More precisely, let $(M^n,g,f,h)$ be an Einstein-type manifold with $h=\frac{R}{n-1}f-\frac{R}{n(n-1)},$  $f=1+F$ and $R$ constant, where F is the potential function in the critical point equation, see \cite{L15}, in this case, we obtain

\begin{corollary} Let $(M^n, g, f,h)$ be an Einstein-type manifold with $R$ constant and $\partial M=\emptyset$. Then $(M^n,g)$ is Einstein if and only if
 $$\Lambda=\frac{1}{2}|\nabla f|^2+hf+c_nRf^2=\frac{1}{2}|\nabla F|^2+\frac{R}{2n(n-1)}F^2$$
 is constant.
\end{corollary}

\begin{acknowledgement}
The first author was partially supported by Brazilian National Council for Scientific and Technological Development (CNPq Grants 403349/2021-4 and 408834/2023-4) and FAPITEC/SE/Brazil.
The second author was partially supported by PPP/FAPEPI/MCT/CNPq, Brazil [Grant: 007/2018], and CNPq/Brazil [Grant: 422900/2021-4] and [Grant:302389/2022-9].
\end{acknowledgement}

\end{document}